\documentclass[11pt,reqno]{amsart}

\usepackage[dvipsnames]{xcolor}
\usepackage{tikz}

\usepackage{stmaryrd}
\usetikzlibrary{matrix,arrows,decorations.pathmorphing}
\usepackage{graphicx}
\usepackage{setspace}
\usepackage{amssymb}
 \usepackage[stable]{footmisc}
\usepackage{amsmath,mathtools}
\usepackage{fullpage}
\usepackage{caption}
\usepackage{amsthm}
\usepackage{mathrsfs}
\usepackage{thmtools}
\usepackage{blkarray}
\usepackage{multirow}
\usepackage{picinpar} 
\usepackage{tikz-cd}
\usepackage{color}
\usepackage{verbatim}
\usepackage{hyperref}
\hypersetup{colorlinks=true, linkcolor=black, citecolor = OliveGreen, urlcolor = magenta}
\usepackage{amssymb}
\usepackage{amsmath}
\usepackage{enumitem,colonequals}

\usepackage{color}
\usepackage{mathrsfs}
\usepackage[all]{xy}
\usepackage{comment}

\usepackage[margin=1.in]{geometry}

\newcommand{\mQ}{\mathbb{Q}}

\newcommand{\mC}{\mathbb{C}}
\newcommand{\mG}{\mathbb{G}}
\newcommand{\mN}{\mathbb{N}}

\newcommand{\mP}{\mathbb{P}}

\newcommand{\cO}{\mathcal{O}}

\newcommand{\wt}[1]{\widetilde{#1}}

\usepackage[OT2,T1]{fontenc}
\DeclareSymbolFont{cyrletters}{OT2}{wncyr}{m}{n}
\DeclareMathSymbol{\Sha}{\mathalpha}{cyrletters}{"58}

\DeclareMathSymbol{\Sha}{\mathalpha}{cyrletters}{"58}

\newcommand{\sO}{{\mathscr O}}

\RequirePackage{mathrsfs} 

\theoremstyle{theorem}
\numberwithin{equation}{subsection}
\newtheorem{thmx}{\text{Theorem}}

\newtheorem{corox}[thmx]{\text{Corollary}}
\newtheorem{theorem}[subsection]{Theorem}
\newtheorem{subtheorem}[subsubsection]{Theorem}

\newtheorem{sublemma}[subsubsection]{Lemma}

\newtheorem{subcorollary}[subsubsection]{Corollary}

\newtheorem{subprop}[subsubsection]{Proposition}

\newtheorem{subconjecture}[subsubsection]{Conjecture}

\numberwithin{equation}{subsection}
\theoremstyle{definition}

\newtheorem{subdefinition}[subsubsection]{\text{Definition}}

\newtheorem{subremark}[subsubsection]{Remark}

\newtheorem{subexample}[subsubsection]{Example}

\theoremstyle{remark}

\numberwithin{equation}{subsubsection} \numberwithin{figure}{section}

\DeclareMathOperator{\an}{an}

 \DeclareMathOperator{\Spec}{Spec}

\DeclareMathOperator{\Hom}{Hom}

\DeclareMathOperator{\ord}{ord}

\DeclareMathOperator{\GL}{GL}

\DeclareMathOperator{\codim}{codim}

\newcommand{\cdef}[1]{\textsf{\textit{#1}}}

\newcommand\GG{\mathbb{G}}
\newcommand\Ga{\GG_\mathrm{a}}

\renewcommand{\leq}{\leqslant}

\renewcommand{\geq}{\geqslant}
\renewcommand{\ge}{\geqslant}

\DeclareMathOperator{\Exc}{Exc}
\DeclareMathOperator{\Ext}{Ext}
\DeclareMathOperator{\uU}{U}
\DeclareMathOperator{\T}{T}
\DeclareMathOperator{\A}{A}
\DeclareMathOperator{\Alb}{Alb}

\makeatletter
\@namedef{subjclassname@1991}{\emph{2020} Mathematics Subject Classification}
\makeatother

\begin{document}

\title{The Green--Griffiths--Lang--Vojta conjecture for commutative algebraic groups with unipotent rank 1}

\author{Jackson S. Morrow}
\address{Jackson S. Morrow \\
	Department of Mathematics\\
	University of North Texas \\
	Denton, TX 76203, USA}
\email{jackson.morrow@unt.edu}

\author{Paul Vojta}
\address{Paul Vojta \\
	Department of Mathematics\\
	University of California, Berkeley \\
	Berkeley, CA 94720, USA}
\email{vojta@math.berkeley.edu}

\begin{abstract}
Let $k$ be algebraically closed field of characteristic zero, let $G$ be a commutative algebraic group over $k$ such that the linear part of $G$ is isomorphic to $\mathbb{G}_a$, and let $X$ be a closed subvariety of $G$. 
We show that the Kawamata locus of $X$ is equal to a Lang-like exceptional locus of $X$, and furthermore, we identify a condition on $X$ that implies that these loci are proper subschemes of $X$.  
Using this, we prove the strong form of the Green--Griffiths--Lang--Vojta conjecture for closed subvarieties of $\mathbb{G}_a \times A$ where $A$ is an abelian variety. 
\end{abstract}

\subjclass
{\href{https://mathscinet.ams.org/msc/msc2020.html?t=14L10\&s=14G05\&btn=Search\&ls=Ct}{14G05}
(\href{https://mathscinet.ams.org/msc/msc2020.html?t=14G05&s=32Q45&btn=Search&ls=Ct}{32Q45}, 
\href{https://mathscinet.ams.org/msc/msc2020.html?t=32Q45&s=14L10&btn=Search&ls=Ct}{14L10})}

\date{\today}
\maketitle


\section{\bf Introduction}
In this work,  we prove algebraic and complex analytic results concerning the structure of certain exceptional loci of closed subvarieties of commutative algebraic groups with unipotent rank one. 

Let $k$ be an algebraically closed field of characteristic zero and let $X$ be a smooth projective $k$-variety. 
A conjecture of Bombieri--Lang\footnote{Bombieri formulated this conjecture in the case of surfaces.} \cite{Lang60} predicts that when $X$ is of general type,  $X$ is arithmetically hyperbolic, which means that there exists a Zariski closed proper subscheme $Z$ of $X$, called the exceptional locus of $X$, such that for any finitely generated subfield $k'$ of $k$, $X$ has only finitely many $k'$-rational points outside of $Z(k')$.  
There are several conjectural descriptions of this subscheme that are related to algebraic, complex analytic, and non-Archimedean analytic characterizations of being of general type, and the conjectures of Green, Griffiths,  and Lang describe the precise relationship between these characterizations
There is also a variant of this conjecture due to the second author concerning integral points on non-projective $k$-varieties which are of logarithmic general type. 
We refer the reader to \cite{JBook} for further discussion on these notions and the precise conjectures. 

Most of the progress on the Lang--Vojta conjecture has focused on the setting where $X$ is a closed subvariety of a commutative algebraic $k$-group $G$.
When $G$ is complete (i.e., $G$ is an abelian $k$-variety), Faltings \cite{FaltingsLang1, FaltingsLang2} proved that the intersection of $X$ with a finitely generated subgroup $\Gamma \subset A(k)$ is a finite union of cosets of subgroups of $\Gamma$. 
The second author \cite{Vojta:IntegralPoints1, Vojta:Integralpoints2} proved this result holds in the more general setting where $G$ is a semi-abelian $k$-variety (i.e., an extension of an abelian $k$-variety by a torus). 
Additionally, McQuillan \cite{McQuillian:DivisionPointsSemiAb} showed that the second author's results can be extended where $\Gamma$ is replaced by any finite rank subgroup of a semi-abelian $k$-variety. 
In each of these settings,  one can use a property of morphisms between semi-abelian $k$-varieties (\autoref{thm:Iitakagrouphom}) to show that the exceptional locus of $X$ is equal to the Kawamata locus of $X$ (\autoref{defn:Kawamatalocus}). 
Furthermore, results of Abramovich, Kawamata, and Ueno \cite{Ueno, Kawamata, Abram} proved that when $X$ is of log-general type,  the Kawamata locus of $X$ is a proper subscheme of $X$, which provides a positive answer to the Lang--Vojta conjectures. 
An important aspect in this setting is that the Lang exceptional locus of $X$ is equal to the Kawamata locus of $X$. 

It is natural to try and extend these kinds of results to a general commutative algebraic $k$-group $G$. 
However, there are various examples of $G$ for which the linear part $H$ of $G$ has $\dim(H)\geq 2$ and is not a torus, such that there exists a closed subvariety $X$ of $G$ which is not a coset of an algebraic subgroup of $G$ and meets a finitely generated subgroup of $G(k)$ in a Zariski dense subset.  We refer the reader to \cite[Examples 1.1 and 1.2]{Ghiocaetal:VariantML} for two such examples.  
These examples illustrate that, in general, the Lang exceptional locus of $X$ cannot equal the Kawamata locus of $X$, and the only possible situation for which such a result could hold is when $G$ is an extension of an abelian $k$-variety by a copy of $\mG_{a}$.  
We will refer to such a $G$ as a commutative algebraic $k$-group with unipotent rank 1 and toric rank 0 (see \autoref{defn:unipotent} for an explanation of this terminology). 

The goal of our work is threefold. 
First, we study the relationship between the exceptional locus and the Kawamata locus for commutative algebraic $k$-group $G$ with unipotent rank 1 and toric rank 0. 
Second, we identify a condition on a subvariety $X$ of $G$ which will guarantee that these loci are proper closed subschemes of $X$, and finally, we use these results to prove the strong form of the Green--Griffiths--Lang--Vojta conjectures for such $X$ when $G\cong \mG_a \times A$. 

\subsection*{Statement of algebraic results}
First, we show that when $G$ has unipotent rank 1 and toric rank 0, the exceptional locus of a closed subvariety $X$ of $G$ is equal to the Kawamata locus of $X$. 

\begin{thmx}\label{thmx:main0}
Let $k$ be an algebraically closed field of characteristic zero, let $G$ be a commutative algebraic $k$-group with unipotent rank 1 and toric rank 0, and let $X$ be a closed subvariety of $G$. 
The Lang-like exceptional locus of $X$ (\autoref{defn:Langlike}) is equal to the Kawamata locus of $X$ (\autoref{defn:Kawamatalocus}). 
\end{thmx}

In the proof of \autoref{thmx:main0}, we identify a condition on $X$, which we call not fibered by subgroups, that allows us to deduce when the Kawamata locus of $X$ is a proper closed subscheme of $X$. 

\begin{thmx}\label{thmx:main1}
Let $k$ be an algebraically closed field of characteristic zero, let $G$ be a commutative algebraic $k$-group with unipotent rank 1 and toric rank 0, and let $X$ be a closed subvariety of $G$. 
If $X$ is not fibered by subgroups (\autoref{defn:notfibered}), then the Kawamata locus of $X$ is a proper closed subscheme of $X$. 
\end{thmx}

As a corollary to \autoref{thmx:main1}, we relate our condition of not fibered by subgroups to a conjectural characterization of being of log-general type. 

\begin{corox}\label{corox:main4}
Let $k$ be an algebraically closed field of characteristic zero, let $G$ be a commutative algebraic $k$-group with unipotent rank 1 and toric rank 0, and let $X$ be a closed subvariety of $G$. 
If $X$ is not fibered by subgroups (\autoref{defn:notfibered}), then $X$ is pseudo-groupless (\autoref{defn:grouplessandpseudo}). 
\end{corox}

\subsection*{Statement of complex analytic results}
Using our above results, we prove the strong form of Green--Griffiths--Lang--Vojta conjectures for closed subvarieties $X$ of $\mG_a \times A$. 
We refer the reader to \autoref{conj:strongNAGGLV} for details on the precise statement of this conjecture. 

\begin{thmx}\label{thmx:main2}
Let $A$ be an abelian $\mC$-variety, and let $X$ be a closed subvariety of $\mG_{a,\mC}\times A$. 
Then, $X^{\an}$ is Brody hyperbolic modulo the analytification of the Lang-like exceptional locus of $X$.
\end{thmx}

\begin{corox}\label{corox:main1}
Let $A$ be an abelian $\mC$-variety, and let $X$ be a closed subvariety of $\mG_{a,\mC}\times A$. 
If $X$ is not fibered by subgroups,  then $X^{\an}$ is pseudo-Brody hyperbolic. 
\end{corox}

\subsection*{Related results}
As mentioned above, there are several results concerning the exceptional loci in abelian and semi-abelian $k$-varieties. 
Despite this, there is a dearth of results in the setting of a commutative algebraic $k$-group with unipotent rank 1 and toric rank 0, and to the authors' knowledge, there are no such results in the setting of a general commutative algebraic $k$-group.

In the case of unipotent rank 1 and toric rank 0,  work of Ghioca--Hu--Scanlon--Zannier \cite{Ghiocaetal:VariantML} proved that for a commutative algebraic $\overline{\mQ}$-group $G\cong \mG_a \times A$ and $X$ a closed subvariety of $G$, if the Bombieri--Lang conjecture holds, then the intersection of $X$ with a finitely generated subgroup $\Gamma \subset G(\overline{\mQ})$ is a finite union of cosets of subgroups of $\Gamma$.  
They also prove an unconditional version of this result when $G$ is an extension of an elliptic curve by a copy of $\mG_a$. 

Our proof of \autoref{thmx:main2} is essentially the Nevanlinna variant of \cite[Theorem 1.4]{Ghiocaetal:VariantML}. While the authors of \textit{loc.~cit.~}need to invoke the Bombieri--Lang conjecture, we are able to overcome this obstacle using a results of Yamanoi concerning entire curves abelian varieties \cite{Yamanoi:HolomorphicAV} and in varieties with maximal Albanese dimension \cite{Yamanoi:MaxAlb}.

We also note that in \cite[Remark 3.1]{Ghiocaetal:VariantML} the authors state what they expect to be true for a Mordell--Lang type theorem for a general commutative algebraic $k$-group $G$, and interestingly, they expect that if the intersection of a closed subvariety $X$ of $G$ with a finite rank subgroup is Zariski dense in $X$, then $X$ should essentially be fibered by subgroups (\autoref{defn:notfibered}).  
Our \autoref{thmx:main1} strongly supports this expectation.  

\subsection*{Organization}
In Section \ref{sec:conventions}, we state our conventions. 
We discuss background on algebraic groups in Sections \ref{sec:prelimsalgebraicKgroups}. 
In Section \ref{sec:hyperbolicity}, we recall various notions of hyperbolicity in the algebraic and complex analytic setting and also introduce the various exceptional loci in our work. 
The proofs of our results occupy the remaining two sections. 
In Section \ref{sec:proofThm0}, we prove our algebraic results, in particular \autoref{thmx:main0} and \autoref{thmx:main1}, and in Section \ref{sec:NAGGLV}, we prove our main complex analytic result, \autoref{thmx:main2}. 

\subsection*{Acknowledgments}
The first author is grateful for the support of the U.S.~National Science Foundation (DMS-2418796) and the Simons Foundation (MPS-TSM-00007985). 
The second author thanks the Institute for Computational and Experimental Research in Mathematics (ICERM) for kind hospitality during Spring 2012, when some of the work leading to this paper took place.
We thank Michel Brion for providing us with two proofs of \autoref{prop:largestabelian}. 
The authors are grateful to the attentive referee for their thorough reading of the manuscript, alerting us to several mistakes in first version of this manuscript, and for also explaining how to extend our previous proof of \autoref{prop:KawamataequalLanglike} to general commutative algebraic groups.

\section{\bf Conventions}
\label{sec:conventions}
In this section, we establish our conventions and recall some background on algebraic groups and Berkovich analytic spaces. 
\subsection{Fields}
\label{subsec:fields}
We will use $k$ to denote an algebraically closed field of characteristic zero, and $\mC$ will denote the complex numbers.

\subsection{Algebraic geometry}
By a $k$-variety, we mean a separated, integral, positive dimensional scheme of finite type over $\Spec(k)$. 
By a subvariety of a $k$-variety,  we mean a positive dimensional closed integral subscheme.  Note that a subvariety is determined by its underlying topological space and we frequently identify both.
By a point of a $k$-variety, we mean a $0$-dimensional closed integral subscheme. 

\subsection{Group varieties}
Algebraic $k$-groups will play an important role in our work, so we recall their definition. 
We say a $k$-scheme $G$ is an \cdef{algebraic $k$-group} if $G$ is a connected group scheme of finite type over $\Spec(k).$ 
Since $k$ is a field of characteristic zero,  \cite[\href{https://stacks.math.columbia.edu/tag/047N}{Tag 047N}]{stacks-project} implies that $G$ is smooth over $\Spec(k)$. 
Two important examples of algebraic $k$-groups in our work will be the \cdef{additive $k$-group} $\mG_{a,k} := \Spec k[t]$ and the \cdef{multiplicative $k$-group} $\mG_{m,k} := \Spec k[t,t^{-1}]$.  
Usually, we will omit the subscript $k$ from our notation as it will be clear from context. 

A complete algebraic $k$-group is called an abelian variety,  which we will write as an \cdef{abelian $k$-variety}.  
Recall that a homomorphism between abelian $k$-varieties $f\colon A \to A'$ of the same dimension is called an \cdef{isogeny} if one of the following equivalent conditions holds:~$f$ is surjective, $f$ is finite and flat, or $\ker(f)$ is finite on $k$-points.

We refer the reader to \cite{Serre:AlgebraicGroups} and \cite{Brion:StructureTheorems} for more background on algebraic $k$-groups.

\subsection{Complex analytic spaces}
Given a $\mC$-variety $X$, we will let $X^{\an}$ denote the complex analytification of $X$. 
We refer the reader to \cite{GrauertFrit:SeveralComplex} and \cite{GrauertRemmert:Coherent} for definitions and standard properties of complex analytic spaces.

\section{\bf Preliminaries on algebraic $k$-groups}
\label{sec:prelimsalgebraicKgroups}
In this section, we discuss commutative algebraic groups and equivariant completions of commutative algebraic groups. 

\subsection{Commutative algebraic groups}
We recall background on the structure of commutative algebraic groups. 
First,  we have Chevalley's theorem concerning the decomposition of algebraic groups of perfect fields. 

\begin{subtheorem}[\protect{\cite{conradChev}}]\label{thm:Chevalley}
Let $G$ be an algebraic $k$-group. Then, $G$ has a unique closed normal subgroup $H$ such that $H$ is a linear $k$-group (i.e., a group subvariety of $\GL_n(k)$) and $G/H$ is an abelian $k$-variety. 
In other words, there exists a short exact sequence of algebraic $k$-groups
\[
0 \to H \to G \to A \to 0
\]
where $A\cong G/H$ is an abelian $k$-variety. 
\end{subtheorem}

We note that $H$ is the maximal linear subgroup of $G$.  Second, we have a theorem of Serre describing the (non-unique) structure of commutative linear algebraic groups over a perfect field.

\begin{subtheorem}[\protect{\cite[Proposition III.7.12 \& Corollary of Proposition VII.7.8]{Serre:AlgebraicGroups}}]\label{thm:commlinear}
A commutative linear algebraic $k$-group is isomorphic to a product 
\[
\mG_{a,k}^{\alpha} \times \mG_{m,k}^{\mu}.
\]
The isomorphism is not in general unique, but $\alpha,\mu \in \mN$ are unique. 
\end{subtheorem}

As an immediate corollary to \autoref{thm:Chevalley} and \autoref{thm:commlinear}, we have the following. 

\begin{subcorollary}\label{corollary:structurecommutative}
Let $G$ be a commutative algebraic $k$-group. 
Then, there exist unique $\alpha,\mu \in \mN$ and a short exact sequence of algebraic $K$-groups 
\begin{equation}\label{eqn:commutativealgebraic}
0 \to \mG_{a,k}^{\alpha} \times \mG_{m,k}^{\mu} \to G \to A \to 0
\end{equation}
where $A \cong G/(\mG_{a,k}^{\alpha} \times \mG_{m,k}^{\mu})$  is an abelian $k$-variety. 
\end{subcorollary}

Using \autoref{corollary:structurecommutative}, we define the following quantities associated to a commutative algebraic $k$-group. 

\begin{subdefinition}\label{defn:unipotent}
Let $G$ be a commutative algebraic $k$-group. 
We say that $G$ has \cdef{unipotent rank $s$ and toric rank $t$} if the short exact sequence from \eqref{eqn:commutativealgebraic} is of the form
\[
0 \to \mG_{a,k}^{s} \times \mG_{m,k}^{t} \to G \to A \to 0
\]
for some $s,t\in \mN$. 
In some cases, we will not specify a unipotent or toric rank, which the reader should take to mean that such a rank is any non-negative integer. 
\end{subdefinition}

\begin{subexample}\label{exam:abeliansemiabelian}
When $G$ has unipotent rank and toric rank equal to 0, $G$ is isomorphic to $A$ and hence $G$ is an abelian $k$-variety. 
When $G$ has unipotent rank 0, $G$ is a semi-abelian variety i.e., an extension of an abelian variety by a (split) torus. 
Note that an abelian variety is necessarily a semi-abelian variety i.e., we do not assume non-zero toric rank in the above sentence. 
\end{subexample}

We recall a result concerning the representability of quotients by normal subgroups. 

\begin{subtheorem}\label{thm:quotientsrepn}
Let $G$ be an algebraic $k$-group, and let $H$ be a normal subgroup of $G$. 
Then, the quotient $G/H$ is a smooth $k$-scheme of finite type which has the unique structure of an algebraic $k$-group such that the map $q\colon G \to G/H$ is a homomorphism.
\end{subtheorem}

\begin{proof}
This is \cite[VIA.3.2]{SGA3}. 
\end{proof}

To conclude,  we prove a result which states that commutative algebraic $k$-group with unipotent rank 1 are isomorphic to certain products over $A$ of simpler commutative algebraic $k$-groups. 

\begin{sublemma}\label{lemma:productoverabelian}
Let $G$ be a commutative algebraic $k$-group with unipotent rank 1 and toric rank $t$ so there exists a short exact sequence of the form
\[
0 \to \mG_a \times \mG_m^t \to G \to A \to 0
\]
where $A$ is an abelian $k$-variety.  
Then, $G$ is isomorphic to a product over $A$ of $t + 1$ commutative algebraic groups which have either unipotent rank 1 and toric rank 0 or unipotent rank 0 and toric rank 1. 
\end{sublemma}

\begin{proof}
First, note that $G/\mG_a$ is a semi-abelian variety, see \autoref{exam:abeliansemiabelian}. 
In this setting, \cite[Lemma 2.2]{Vojta:IntegralPoints1} shows that $G' = G/\mG_a$ is the product over $A$ of $t$ commutative algebraic groups of unipotent rank 0 and toric rank 1. 

The remainder of the lemma is proved by induction on $t$. 
The case of $t = 0$ is clear.
For the inductive set, we apply the contravariant exact sequence in $\Ext(\cdot,\mG_a)$ to the short exact sequence
\[
0\to \mG_m \to G'' \to G''' \to 0
\]
where $G''$ and $G'''$ are semi-abelian varieties.  
Note that this sequence is stricly exact in the sense of \cite{Serre:AlgebraicGroups}, and hence this yields the sequence
\[
\Hom(\mG_m,\mG_a) \to \Ext(G''',\mG_a) \to \Ext(G'',\mG_a) \to \Ext(\mG_m,\Ga).
\]
By \cite[Proof of Proposition VII.6.7]{Serre:AlgebraicGroups}, we have that 
$\Hom(\mG_m,\mG_a) = 0$ and $\Ext(\mG_m,\Ga) = 0$, and therefore, $ \Ext(G''',\mG_a) \cong \Ext(G'',\mG_a) $. 
Now the induction hypothesis tells us that $\Ext(A,\mG_a) $ is in bijection with $\Ext(G',\mG_a)$, and therefore the result follows. 
\end{proof}

\subsection{Equivariant completion of commutative algebraic groups}
\label{subsec:equivariantcompletion}
In this section, we recall a construction of Serre \cite{Serre:Appendix} concerning an equivariant completion of a commutative algebraic group. 
First, we recall this notion. 

\begin{subdefinition}
Let $G$ be an algebraic $k$-group
An \cdef{equivariant completion} of $G$ is a complete $k$-scheme $X$ equipped with an action of $G$ such that there is an open equivariant immersion $G\hookrightarrow X$ with schematically dense image. 
\end{subdefinition}

Let $G$ be a commutative algebraic $k$-group, so $G$ fits into the short exact sequence of algebraic $k$-groups
\[
0\to H \to G \to A \to 0
\]
where $H \cong \mG_a^\mu \times \mG_m^t$ for some $\mu,t \in \mN$ and $A$ is an abelian $k$-variety.
In \cite[Section 2]{Serre:Appendix},  Serre constructed an equivariant completion of $G$ by using that for given a projective $k$-variety $P$ with a $H $-action, every open $H$-equivariant immersion $H \to P$ induces in a natural way a completion $\overline{G}$ of $G$, namely the fiber bundle $\overline{G} = G \times^H P$ with fiber $P$ over $A$, associated to the $H$-principal bundle $G\to A$ with respect to the given $H$-action of $P$. 

In the setting where $G$ has unipotent rank 1 and toric rank 0, we may take the above $P$ to be $\mP^1_k$, and hence $\overline{G}$ is isomorphic to $\mP^1 \times U$ locally over $A$. 
Moreover, $\overline{G}$ is a projective bundle over $A$, and hence there exists a rank 2 vector sheaf $\mathcal{E}$ on $A$ for which $\overline{G} = \mP(\mathcal{E})$. 

To conclude our discussion on these completions, we record a fact about the conormal sheaf of the divisor at infinity of $\overline{G}$ for future use.

\begin{sublemma}\label{lemma:conormaltrivial}
Let $G$ be a commutative algebraic $k$ group with unipotent rank 1 and toric rank 0, let $\overline{G}$ be the equivariant completion of $G$ as defined in \cite[Section 2]{Serre:Appendix}, and let $D = \overline{G}\setminus G$ be the divisor at infinity of $\overline{G}$. 
The conormal sheaf $\mathcal{C}_{D/\overline{G}}$ is trivial.
\end{sublemma}
\begin{proof}
This follows from Serre's construction of $\overline{G}$. 
Indeed, the gluing of open subsets $\mP^1 \times U$ over $A$ involves only translations, which does not affect the tangent vectors at infinity.  
In particular, this implies that $\mathcal{C}_{D/\overline{G}}$ is trivial. 
\end{proof}

\section{\bf Notions of hyperbolicity}
\label{sec:hyperbolicity}
In this section, we will describe various types of exceptional loci, algebraic notions of hyperbolicity, and complex notions of hyperbolicity.

\subsection{Exceptional sets of closed subvarieties of commutative algebraic groups}
To begin, we define the Lang exceptional set, a variant of the Lang exceptional set, and the Kawamata locus. 
We refer the reader to Section \ref{sec:conventions} for our conventions concerning $k$-varieties and algebraic $k$-groups. 

\begin{subdefinition}\label{defn:Langexceptional}
Let $X$ be a complete $k$-variety. 
The \cdef{Lang exceptional locus of $X$} is the Zariski closure of the images of all non-constant rational maps from algebraic $k$-groups to $X$. 
\end{subdefinition}

\begin{subremark}
Since $X$ is complete and any algebraic $k$-group $G$ is smooth,  it is well-known that the indeterminacy locus of any non-constant rational map $G \dashrightarrow X$ has codimension $\geq 2$ i.e., the rational map is defined in codimension 1.
Furthermore, when $X$ is pure i.e., contains no rational curves, we can assume that the rational maps are in fact morphisms by \cite[Lemma 3.5]{JKam}. 
This condition holds for example when $X$ is a closed subvariety of an abelian $k$-variety. 
\end{subremark}

\autoref{defn:Langexceptional} motivates the following definition, which is an analogue of \autoref{defn:Langexceptional} in the non-complete setting. 

\begin{subdefinition}\label{defn:Langlike}
Let $X$ be a $k$-variety. 
The \cdef{Lang-like exceptional locus $\Exc'(X)$ of $X$} is the Zariski closure of the union of the images of all non-constant, rational maps $G\dashrightarrow X$ defined outside of a subset of codimension $\ge2$ where $G$ is an algebraic $k$-group.  
\end{subdefinition}

\begin{sublemma}\label{lemma:LanglikeGmbigopenAVtest}
Let $X$ be a $k$-variety. 
The Lang-like exceptional locus $\Exc'(X)$ of $X$ is equal to the Zariski closure of the union of all non-constant morphisms from $\mG_m$ or dense open subschemes $U$ of abelian varieties $A$ with $\codim(A\setminus U) \geq 2$ to $X$. 
\end{sublemma}

\begin{proof}
By definition, the latter set is contained in $\Exc'(X)$, and the reverse inclusion is precisely \cite[Lemma 3.16]{JavanpeykarXie:Pseudogroupless}. 
\end{proof}

We will be primarily interested in the setting where $X$ is a closed subvariety of an algebraic $k$-group. 
In this case, we may simplify the description of the Lang-like exceptional set.

\begin{sublemma}\label{lemma:LanglikeGmAV}
Let $X$ be a closed subvariety of an algebraic $k$-group. 
The Lang-like exceptional locus $\Exc'(X)$ of $X$ is equal to the Zariski closure of the union of all non-constant morphisms from $\mG_m$ or abelian varieties $A$ to $X$. 
\end{sublemma}

\begin{proof}
This follows from \autoref{lemma:LanglikeGmbigopenAVtest} and the fact that rational maps from smooth $k$-varieties to algebraic $k$-groups which are defined in codimension 1 uniquely extend to $k$-morphisms \cite[Theorem 4.4.1]{BLR}.  
\end{proof}

When $X$ is a closed subvariety of an algebraic $k$-group, we can use \autoref{lemma:LanglikeGmAV} to write
\begin{equation}\label{eqn:Langlikebreakdown}
\Exc'(X) = \Exc'_{\T}(X) \cup \Exc'_{\A}(X)
\end{equation}
where $\Exc'_{\T}(X)$ and $\Exc'_{\A}(X)$ are the Zariski closures of the unions of images of morphism from $\mG_m$ to $X$ and from abelian $k$-varieties to $X$, respectively.  
We will refer to $\Exc'_{\T}(X)$ and $\Exc'_{\A}(X)$ as the \cdef{toric} and \cdef{abelian exceptional locus} of $X$, respectively. 

In addition to the Lang-like exceptional locus, we define the Kawamata locus. 

\begin{subdefinition}\label{defn:Kawamatalocus}
Let $X$ be a closed subvariety of an algebraic $k$-group $G$. 
The \cdef{Kawamata locus of $X$} is the union $Z(X)$ of all non-trivial translated group subvarieties of $G$ contained in $X$. 
\end{subdefinition}

It is known that the Kawamata locus of $X$ is closed by work of Abramovich. 

\begin{subtheorem}[\protect{\cite[Theorem 1]{Abram}}]\label{thm:Kawamataclosed}
Let $X$ be a closed subvariety of an algebraic $k$-group $G$. 
Then $Z(X)$ is closed. 
\end{subtheorem}

Furthermore, if $G$ is semi-abelian, the Kawamata locus of $X$ equals the Lang-like exceptional locus of $X$, as a consequence of the following theorem of Iitaka.  

\begin{subtheorem}[\protect{\cite[Theorem 2]{IitakaLogForms}}]\label{thm:Iitakagrouphom}
Any $k$-morphism from one semi-abelian variety to another is a translate of a group $k$-homomorphism. 
\end{subtheorem}

When $G$ is a commutative algebraic $k$-group and $X$ is a closed subvariety of $G$, we may write
\begin{equation}\label{eqn:Kawamata}
Z(X) = Z_{\uU}(X) \cup Z_{\T}(X) \cup Z_{\A}(X)
\end{equation}
where $Z_{\uU}(X),Z_{\T}(X),$ and $Z_{\A}(X)$ are the unions of translated group subvarieties of $G$ isomorphic to $\mG_a$, $\mG_{m}$, and non-trivial abelian varieties contained in $X$, respectively. 
We will refer to $Z_{\uU}(X),Z_{\T}(X),$ and $Z_{\A}(X)$ as the \cdef{unipotent, toric,} and \cdef{abelian Kawamata locus} of $X$, respectively.

\subsection{Algebraic notions of hyperbolicity}
Next, we recall various notions of hyperbolicity in the algebraic setting and relate them to the loci described above.  

First, we define the notion of groupless and pseudo-groupless, which have appeared in works of \cite{LangDiophantine2}, the second author \cite{VojtaLangExc}, and in \cite{JKam, JavanpeykarXie:Pseudogroupless}. 

\begin{subdefinition}\label{defn:grouplessandpseudo}
Let $X$ be a $k$-variety. 
\begin{enumerate}
\item $X$ is \cdef{groupless} if for every algebraic $k$-group $G$, every $k$-morphism $G\to X$ is constant. 
\item Let $\Delta\subset X$ be a closed subscheme.  $X$ is \cdef{groupless modulo $\Delta$} if for every algebraic $k$-group and every dense open subscheme $U\subseteq G$ with $\codim(G\setminus U)\geq 2$, every non-constant $k$-morphism $U \to X$ factors over $\Delta$.
\item $X$ is \cdef{pseudo-groupless} if there exists a proper closed subscheme $\Delta\subset X$ such that $X$ is groupless modulo $\Delta$.
\end{enumerate}
\end{subdefinition}

Roughly speaking, it suffices to test being groupless and pseudo-groupless on $\mG_{m,k}$ and big open subschemes of abelian $k$-varieties. 

\begin{sublemma}\label{lemma:testgroupless}
Let $X$ be a $k$-variety, and let $\Delta\subset X$ be a closed subscheme. 
\begin{enumerate}
\item $X$ is groupless if and only if every $k$-morphism from either $\mG_{m,k}$ or an abelian $k$-variety is constant. 
\item $X$ is groupless modulo $\Delta$ if and only if every non-constant $k$-morphism $\mG_{m,k}\to X$ factors over $\Delta$ and for every abelian $k$-variety $A$ and every dense open subscheme $U \subseteq A$ with $\codim(A\setminus U)\geq 2$, every non-constant $k$-morphism $U \to X$ factors over $\Delta$. 
\end{enumerate}
\end{sublemma}

\begin{proof}
This first statement is \protect{\cite[Lemma 2.5]{JKam}, and the second is \cite[Lemma 3.16]{JavanpeykarXie:Pseudogroupless}}. 
\end{proof}

\begin{subcorollary}\label{lemma:grouptestgroup}
Let $X$ be a closed subvariety of an algebraic $k$-group, and let $\Delta\subset X$ be a closed subscheme. 
Then, $X$ is groupless modulo $\Delta$ if and only if every non-constant $k$-morphism $\mG_{m,k}\to X$ factors over $\Delta$ and for every abelian $k$-variety $A$, every non-constant $k$-morphism $A \to X$ factors over $\Delta$. 
\end{subcorollary}

\begin{proof}
This is \autoref{lemma:testgroupless} combined with \cite[Theorem 4.4.1]{BLR}. 
\end{proof}

\begin{subremark}
When $X$ is a complete $k$-variety, one can prove stronger versions of the statements in \autoref{lemma:testgroupless}. We refer the reader to \cite[Lemma 2.7]{JKam} and \cite[Remarks 3.2 \& 3.3]{JavanpeykarXie:Pseudogroupless} for details.  
\end{subremark}

When $X$ is a closed subvariety of a semi-abelian $k$-variety, we have a precise relationship between $X$ being pseudo-groupless and $X$ being of log-general type. 

\begin{subtheorem}[\protect{\cite{Ueno, Kawamata, Abram}}]\label{thm:semiabsetting}
Let $X$ be a closed subvariety of a semi-abelian $k$-variety. 
\begin{enumerate}
\item $X$ is groupless modulo $Z(X)$.
\item $X \neq Z(X)$ if and only if $X$ is of log-general type if and only if $X$ is pseudo-groupless. 
\end{enumerate}
\end{subtheorem}

For our purposes, we will need another definition, which will be better suited for our methods. 

\begin{subdefinition}\label{defn:notfibered}
Let $X$ be a closed subvariety of a commutative algebraic $k$-group $G$. 
We say that $X$ is \cdef{fibered by subgroups} if there exists a non-trivial subgroup $H$ of $G$ such that $X$ is the pullback of a closed subset of $G/H$. Note that the quotient $G/H$ exists as an algebraic $k$-group by \autoref{thm:quotientsrepn}. 
\end{subdefinition}

%


\subsection{Complex notions of hyperbolicity}
To conclude our discussion on hyperbolicity,  we recall the notions of complex hyperbolicity. 
We refer the reader to \cite[Section 8]{JBook} for further discussion. 

\begin{subdefinition}\label{defn:KBrody}
Let $X$ be a $\mC$-variety. 
\begin{enumerate}
\item $X$ is \cdef{Brody hyperbolic} if every holomorphic map $\mC \to X^{\an}$ is constant. 
\item Let $\Delta\subseteq X$ be a closed subscheme.  $X$ is \cdef{Brody hyperbolic modulo $\Delta$} if every holomorphic non-constant map $\mC \to X^{\an}$ factors over $\Delta^{\an}$. 
\item $X$ is \cdef{pseudo-Brody hyperbolic} if there exists a proper closed subscheme $\Delta$ of $X$ such that $X$ is Brody hyperbolic modulo $\Delta^{\an}$. 
\end{enumerate}
\end{subdefinition}

Similar to the algebraic setting, Noguchi \cite{Noguchi:LLDSemiAb} proved there is a precise relationship to $X$ being pseudo-groupless and $X$ being pseudo-Brody hyperbolic in the setting where $X$ is a closed subvariety of a semi-abelian $K$-variety. 
For our purposes, we recall a result of Yamanoi which proves that such a relationship persists when $X$ has maximal Albanese dimension. 

\begin{subtheorem}[\protect{\cite[Corollary 1]{Yamanoi:MaxAlb}}]\label{thm:semiabsettingNA}
Let $X$ be a smooth projective $\mC$-variety of maximal Albanese dimension i.e., the Albanese map $\alpha_X\colon X \to \Alb(X)$ is generically finite onto its image $\alpha(X)$. 
If $X$ is of general type, then $X$ is pseudo-Brody hyperbolic. 
\end{subtheorem}

We conclude by stating the strong form of the Green--Griffiths--Lang--Vojta conjecture.

\begin{subconjecture}[Strong form of the Green--Griffiths--Lang--Vojta conjecture]
\label{conj:strongNAGGLV}
Let $X$ be a $\mC$-variety. 
The following are equivalent:
\begin{enumerate}
\item $X$ is of log-general type,
\item $X$ is pseudo-groupless (\autoref{defn:grouplessandpseudo}.(3)), and
\item $X$ is pseudo-Brody hyperbolic (\autoref{defn:KBrody}.(3)).
\end{enumerate}
\end{subconjecture}

\section{\bf The Kawamata locus and the Lang exceptional locus}
\label{sec:proofThm0}
In this section, we prove our main algebraic results, namely \autoref{thmx:main0} and \autoref{thmx:main1}. 

\subsection{Proof of \autoref{thmx:main0}}
The proof of \autoref{thmx:main0} will follow from three results concerning each part of the decomposition of the Kawamata locus \eqref{eqn:Kawamata}.  
While \autoref{thm:Kawamataclosed} tells us that the Kawamata locus is closed, we will need to analyze the structure of the Kawamata locus further to understand when it is a \textit{proper} closed subscheme of $X$.

\begin{sublemma}\label{lemma:unipotentKLclosed}
Let $G$ be a commutative algebraic $k$-group with unipotent rank 1, and let $X$ be a closed subvariety of $G$. 
The unipotent Kawamata locus $Z_{\uU}(X)$ of $X$ is closed, and in particular, if $X$ is not fibered by subgroups, then $Z_{\uU}(X)$ is a proper closed subscheme of $X$. 
\end{sublemma}

\begin{proof}
First, note that the quotient map $\pi\colon G \to G/\mG_{a}$ is smooth. 
Indeed, the quotient map is faithfully flat with smooth fibers and hence smooth by  \cite[\href{https://stacks.math.columbia.edu/tag/01V8}{Tag 01V8}]{stacks-project}. 
Since $\pi$ is smooth, it is open, and therefore, $\pi(G\setminus X)$ is open in $G/\mG_{a}$. 
As every morphism $\mG_a \to G/\mG_a$ is constant, we have that $Z_{\uU}(X)$ is equal to $\pi^{-1}((G/\mG_{a})\setminus \pi(G\setminus X))$, and hence is a closed subset of $X$.

The latter claim follows from the definition of not fibered by subgroups (\autoref{defn:notfibered}). 
\end{proof}

\begin{subremark}
We remark that the second statement of \autoref{lemma:unipotentKLclosed} need not hold for $G$ of unipotent rank $>1$. 
As an example, one may consider a ruled variety where the rulings vary with points e.g., $X = V(z-xy) \subset \mG_a^3$.  
In this example, we have that $Z_{\uU}(X) = X$ since every point lies on a translated $\mG_a$ contained in $X$,  and one can show that $X$ is not fibered by subgroups by checking that the stabilizer of $X$ in $\mG_a^3$ is trivial.  
\end{subremark}

\begin{sublemma}\label{lemma:toricKawamata}
Let $G$ be a commutative algebraic $k$-group with unipotent rank 1 and toric rank 0, and let $X$ be a closed subvariety of $G$. 
The toric Kawamata locus $Z_{\T}(X)$ of $X$ is empty, and the toric exceptional locus $\Exc'_{\T}(X)$ is equal to the unipotent Kawamata locus $Z_{\uU}(X)$ of $X$. 
\end{sublemma}

\begin{proof}
Since $G$ has toric rank 0, there are no subgroups of $G$ isomorphic to a torus, and hence the toric Kawamata locus $Z_{\T}(X)$ of $X$ is empty. 

For the latter claim, we first note that the unipotent Kawamata locus $Z_{\uU}(X)$ of $X$ is contained in the toric exceptional locus $\Exc'_{\T}(X)$. Indeed, if $G' \subseteq Z_{\uU}(X)$, then $G'$ is a translated subgroup of $G$ which is isomorphic to $\mG_a$ and contained in $X$. 
Note that $\mG_m$ is an open subscheme of $\mG_a$, and hence $G'$ is contained in the Zariski closure of a non-constant morphism from $\mG_m \to X$. 
Thus, it suffices to prove that $\Exc'_{\T}(X)$ is contained in $Z_{\uU}(X)$ i.e., we need to show that the Zariski closure of the union of the images of all morphisms $\varphi\colon \mG_m \to X$ is contained in $Z_{\uU}(X)$. 
Consider the composition of morphisms $\mG_m \to X \subset G \to A$ where $X \subset G$ is simply the inclusion map.
Since $\mG_m$ is smooth, the morphism $\mG_m \to A$ extends to a morphism from $\mP^1 \to A$, and moreover, this map must be constant as $A$ has no rational curves. 
This implies that $\varphi(\mG_m)$ is in fact contained in the translate of the kernel of $G \to A$, which by assumption is $\mG_a$, and so $\varphi(\mG_m)$ is contained in a translate of $\mG_a$. 
Since $\varphi(\mG_m)$ is already contained in $X$, this implies that $\varphi(\mG_m)$ is contained in $Z_{\uU}(X)$.
Since $\varphi$ was arbitrary, we have that the union of the images $\varphi(\mG_m)$ as $\varphi$ ranges over all morphisms from $\mG_m \to X$ are contained in $Z_{\uU}(X)$. 
By \autoref{lemma:unipotentKLclosed}, $Z_{\uU}(X)$ is closed and hence the Zariski closure of the union of the images $\varphi(\mG_m)$ is contained in $Z_{\uU}(X)$ i.e., $\Exc'_{\T}(X) \subseteq Z_{\uU}(X)$. 
\end{proof}

\begin{subprop}\label{prop:KawamataequalLanglike}
Let $G$ be a commutative algebraic $k$-group, and let $X$ be a closed subvariety of $G$. 
The abelian Kawamata locus $Z_{\A}(X)$ is equal to the abelian Lang-like exceptional locus $\Exc'_{\A}(X)$. 
\end{subprop}

\begin{proof}
By definition, we have that $Z_{\A}(X) \subseteq \Exc'_{\A}(X)$.
To prove the reverse inclusion, it suffices to prove that for any abelian $k$-variety $A'$, the image of any $k$-morphism $\varphi\colon A'\to G$ is the translate of an algebraic subgroup.  
Indeed, by \cite[\href{https://stacks.math.columbia.edu/tag/0AH6}{Tag 0AH6}]{stacks-project}, we have that the image $\varphi(A')$ will be complete, and a complete algebraic group is an abelian $k$-variety. 

First, we show how to reduce to the setting where $G$ has unipotent rank $n$ and toric rank 0. 
Write $G$ as in \autoref{defn:unipotent} i.e., 
\[
0 \to \mG_a^n \times \mG_m^t \to G \to A \to 0
\]
where $A$ is an abelian $k$-variety. 
Since the linear part of $G$ is split, we may also identify $G$ with the short exact sequence of commutative algebraic groups
\[
0 \to \mG_a^n \to G \to G' \to 0
\]
where $G'$ is a semi-abelian $k$-variety with toric rank $t$. 
By \autoref{thm:Iitakagrouphom},  we have that the $k$-morphism $A' \to G \to G'$ is the translate of a group $k$-homomorphism, and in particular the image is a translate of an abelian $k$-subvariety $A''$ of $G'$. 
Note that the pre-image of the translate of $A''$ in $G$ is a translate of a closed subgroup $G''$ of $G$, and we also have that $G''$ is a commutative algebraic group of unipotent rank $n$ and toric rank 0. 
Since the image $\varphi(A')$ is contained in the translate of $G''$ by construction, we have reduced our claim to the setting of unipotent rank $n$ and toric rank 0. 

By our reduction, it suffices to consider the case where $G$ is of the form
\begin{equation}\label{eqn:SES}
0 \to \mG_a^n \to G \to A \to 0
\end{equation}
where $A$ is an abelian $k$-variety of dimension $g$. 
Let $\pi\colon G\to A$ denote the quotient map. 
We divide the proof into three cases.\\

\noindent
\textit{Case 1}.~We first assume that the map $\pi$ admits a section $\sigma\colon A \to G$, and after possibly translating we may and do assume that $\sigma(0) = 0$, where $0$ denotes the identity element in the corresponding algebraic group.  
We claim that $\sigma(A)$ is the translate of a subgroup. 

Using standard reductions, we may and do work over $\mC$. 
After considering universal covers, \autoref{eqn:SES} lifts to a short exact sequence of vector spaces, and we may also lift the section $\sigma$ to a section $\hat{\sigma}\colon \mC^g \to \mC^{g+n}$ of universal covers. 
To summarize, we have the diagram
\[
\begin{tikzcd}
0 \arrow{r} & \mC^n \arrow{r} & G^{\an} \arrow{r}{\pi} & A^{\an} \arrow[bend right, swap]{l}{\sigma}\arrow{r} & 0 \\
0 \arrow{r} & \mC^n \arrow{r}\arrow[equal]{u} & \mC^{g+n} \arrow{u}{\phi} \arrow{r}{q} & \mC^g \arrow[bend left]{l}{\hat{\sigma}}\arrow{u}{u} \arrow{r} & 0
\end{tikzcd}
\]
with the following relations
\begin{align}
\pi \circ \sigma &= \text{id}_A \label{eqn:r1}\\
\pi \circ \phi &= u \circ q \label{eqn:r2}\\
\phi\circ \hat{\sigma} &= \sigma \circ u \label{eqn:r3}.
\end{align}

We claim that for all $\lambda\in \ker(u)$, 
\begin{equation}\label{eqn:claim1}
\hat{\sigma}(z + \lambda) = \hat{\sigma}(z) + \hat{\sigma}(\lambda).
\end{equation}
First, note that for $\lambda \in \ker(u)$, we can use \autoref{eqn:r3} and the fact that $\phi$ is a group homomorphism to deduce that $\hat{\sigma}(z+\lambda) - \hat{\sigma} \in \ker(\phi)$. 
Again, for fixed $\lambda\in \ker(u)$, we can consider the holomorphic map
\[
c_{\lambda}\colon \mC^g \to \ker(\phi), \quad z \mapsto \hat{\sigma}(z+\lambda) - \hat{\sigma}(z).
\]
Since $\ker(\phi)$ is discrete, we have that $c_{\lambda}(z) = c_{\lambda}(0)$ for all $z\in \mC^g$, and hence 
\[
\hat{\sigma}(z+\lambda) - \hat{\sigma}(z) = \hat{\sigma}(\lambda) - \hat{\sigma}(0). 
\]
After possibly translating $\hat{\sigma}$ by $-\hat{\sigma}(0)$ (which lies in $\ker(\phi)$ by \autoref{eqn:r3}), we may and do assume that $\hat{\sigma}(0) = 0$, and hence we have proved our desired claim. 

Next, we claim that $\hat{\sigma}\colon \mC^g \to \mC^{g+n}$ is linear. 
\autoref{eqn:claim1} says that $d\hat{\sigma}_{z+\lambda} = d\hat{\sigma}_z$ for all $\lambda\in \ker(u)$, and hence the holomorphic map
\[
d\hat{\sigma}\colon \mC^g \to \Hom_{\mC}(\mC^g,\mC^{g+n}), \quad z\mapsto d\hat{\sigma}_z
\]
is $\ker(u) = \Lambda$-invariant. 
As such, $d\hat{\sigma}$ descends to a holomorphic $\overline{d\hat{\sigma}}$ on the compact complex manifold $\mC^g/\Lambda$. 
The compactness of $\mC^g/\Lambda$ forces $\overline{d\hat{\sigma}}$ to be constant, and moreover, we have that $d\hat{\sigma}(z) = L$. 
Using that $\hat{\sigma}(0) = 0$, one can also deduce that $\hat{\sigma}$ is linear. 

To conclude Case 1, one can check that $\sigma$ is a group homomorphism using \autoref{eqn:r3}, surjectivity of $u$, and the fact that $\hat{\sigma}$ is linear.  
This implies that $\sigma(A^{\an}) \subset G^{\an}$ is an analytic subgroup of $G^{\an}$.
A standard Nagata compactification plus GAGA argument implies that $\sigma(A^{\an})$ is actually an algebraic group of $G^{\an}$, and hence we have that $\sigma(A)$ is an algebraic group of $G$ as desired. \\

%

\noindent
\textit{Case 2}.~Suppose that $\varphi\colon A' \to G$ is any morphism and that $\pi\circ \varphi\colon A' \to A$ is finite. 
By \autoref{thm:Iitakagrouphom}, $\pi\circ \varphi$ is the translate of a group homomorphism, and hence after considering the image of this morphism and translating, we may and do assume that $\pi\circ \varphi \colon A' \to A$ is finite and surjective.  
By Case 1, we have that the map $A' \to G \times_A A'$ is the translate of a group homomorphism, and hence the same holds for $A' \to G \times_A A' \to G$ since $G \times_A A' \to G$ is a group homomorphism. 
Since the kernel of $A'\to A$ is finite by assumption and $\mG_a^n$ has no non-trivial finite algebraic subgroups,  we have that after a possible translation, the kernel of $A'\to A$ equals the kernel of $A'\to G$.  
Let $K$ denote this common kernel.
Since $\pi\circ \varphi$ is surjective, we have that $A'/K \cong A$ and $A' \to G$ factors through the quotient $A'/K \to G$. 
The composition of the isomorphism $A\cong A'/K$ with this quotient produces a section of the map $\pi\colon G\to A$. 
Moreover, Case 1 implies that the image of $\varphi\colon A' \to G$ is the translate of an algebraic subgroup of $G$.\\

\noindent
\textit{Case 3}.~
Now for the general setting, let $\varphi\colon A'\to G$ be
a $k$-morphism.  
By applying a translation to $G$, we may assume that
$\varphi(0)=0$.  We now aim to show that the image of $\varphi$ is a subgroup of $G$.
By Poincar\'e reducibility \cite[Section 18, Corollary 1]{MumAb}, there exist
simple abelian $k$-varieties $B_1,\dots,B_n$ (repeated as necessary)
such that $A'$ is isogenous to the product $B_1\times\cdots\times B_n$.
Let $\psi\colon\prod B_i\to A'$ denote the above isogeny,
and let $\gamma=\pi\circ\varphi\circ\psi\colon\prod B_i\to A$.
By \autoref{thm:Iitakagrouphom}, $\gamma$ is a group homomorphism.

Let $A_0$ be the image of $\gamma$; it is an abelian subvariety of $A$.
After replacing $A$ with $A_0$ and $G$ with $\pi^{-1}(A_0)$, we may assume
that $A=A_0$; i.e., $\pi\circ\varphi\colon A'\to A$ is surjective.

For subsets $I\subseteq\{1,\dots,n\}$, let $B_I$ denote the kernel of
the projection $\prod B_i\to\prod_{i\notin I} B_i$.
Choose $I$ minimal such that $\gamma_I:=\gamma\big|_{B_I}\colon B_I\to A$
is still surjective.  We claim that $\gamma_I$ is finite.
Indeed, if $\gamma_I$ is not finite, then $\dim\ker\gamma_I>0$, so there
exists some $i\in I$ such that $\dim\pi_{Ii}(\ker\gamma_I)>0$, where
$\pi_{Ii}\colon B_I\to B_i$ is the projection.  Since $B_i$ is simple,
this implies that $\pi_{Ii}(\ker\gamma_I)=B_i$, and then a little group
theory gives that $\gamma_{I\setminus\{i\}}$ has the same image as $\gamma_I$,
contradicting the choice of $I$.

Then, since $\gamma_I\colon B_I\to A$ is finite, it follows from Case 2
that the image of $(\varphi\circ\psi)\big|_{B_I}\colon B_I\to G$
is a subgroup of $G$ (since it contains $0$).  This subset (or subscheme)
is contained in the image of $\varphi\circ\psi$, and the latter equals
the image of $\varphi$ since $\psi$ is surjective.

Now we have that the images $\varphi(A')\supseteq\varphi(\psi(B_I))$ in $G$
are both proper over $k$, hence over $A$, and are also affine over $A$
(since $G\to A$ is affine and the subsets are closed), so both of these
images are finite over $A$.  Since both images dominate $A$, they both
have the same dimension (equal to $\dim A$).  Also, since $A'$ and $B_I$
are irreducible, so are their images in $G$.  Therefore they must be equal.

Thus $\varphi(A')$ is a subgroup of $G$ (since $\varphi(\psi(B_I))$ is),
as was to be shown.
\end{proof}

We are now in a position to prove \autoref{thmx:main0}. We recall the statement of this result. 

\begin{subtheorem}[= \autoref{thmx:main0}]
Let $k$ be an algebraically closed of characteristic zero, let $G$ be a commutative algebraic $k$-group with unipotent rank 1 and toric rank 0, and let $X$ be a closed subvariety of $G$. 
The Lang-like exceptional locus of $X$ is equal to the Kawamata locus of $X$. 
\end{subtheorem}

\begin{proof}
This follows from the decomposition of the Lang-like exceptional locus \eqref{eqn:Langlikebreakdown} and the Kawamata locus \eqref{eqn:Kawamata} and \autoref{lemma:toricKawamata} and \autoref{prop:KawamataequalLanglike}. 
\end{proof}

As a corollary of \autoref{prop:KawamataequalLanglike}, we have the following. 
\begin{subcorollary}
Let $G$ be an algebraic $k$-group of the form 
\[
0 \to \mG_a \to G \to A \to 0
\] 
where A is a simple abelian $k$-variety. 
If there exists a non-constant $k$-morphism from an abelian $k$-variety $A’ \to G$, then $G \to A$ admits a regular section, and hence $G$ is isomorphic to $\mG_a \times A$. 
Moreover, the equivariant completion $\overline{G}$ of G defined in Subsection \ref{subsec:equivariantcompletion} is isomorphic to $\mP(\cO_A^2)$.
\end{subcorollary}
\begin{proof}
When $A$ is simple, the proof of \autoref{prop:KawamataequalLanglike}, in particular Cases 2 and 3, implies that the image of $A’ \to G$ is a regular section of $G \to A$ which is disjoint from $G_{\infty}$. Our claims readily follow. 
\end{proof}

\subsection{Proof of \autoref{thmx:main1}}
To prove \autoref{thmx:main1}, we will first show that for any $k$-group $G$, there exists a largest abelian $k$-subvariety. 
In this setting, a $k$-group is a group scheme of finite type over $\Spec(k)$, in particular, we do not assume connectedness. 
We thank Michel Brion for providing us with two proofs of this statement. 

\begin{subprop}\label{prop:largestabelian}
Every $k$-group $G$ has a largest abelian $k$-subvariety.
\end{subprop}

\begin{proof}[Proof 1 of \autoref{prop:largestabelian}]
Every abelian $k$-subvariety $A$ of $G$ is contained in the
connected component of the neutral element of $G$, and hence
we may assume $G$ connected. 
By \cite[Proposition 3.3.4]{Brion:StructureTheorems}, we have that 
$A$ is contained in the center of $G$. 
As such, for any abelian $k$-subvariety $B$ of $G$, the image of the multiplication map $A \times B \to G$ is an
abelian $k$-subvariety, namely the quotient of $A \times_k B$ by a closed subgroup isomorphic to $A \cap B$. 
Taking for $A$ an abelian $k$-subvariety of maximal dimension yields the statement. 
\end{proof}

\begin{proof}[Proof 2 of \autoref{prop:largestabelian}]
First, we claim that the class of
abelian $k$-varieties is stable under group extensions.
Indeed, the class of smooth $k$-groups is stable
under such extensions \cite[Proposition 1.62]{Milne:AlgebraicGroups} and the same result holds for the class of
connected algebraic $k$-groups \textit{loc.~cit.~}Proposition 5.59. 
Finally, the class of proper algebraic $k$-groups is stable under
group extensions, since proper morphisms are stable under
composition and faithfully flat descent \cite[\href{https://stacks.math.columbia.edu/tag/01W3}{Tag 01W3} \& \href{https://stacks.math.columbia.edu/tag/02L1}{Tag 02L1}]{stacks-project}. 
To conclude, one argues as in \cite[Lemma 3.1.4]{Brion:StructureTheorems}
using the stability of abelian varieties under products, quotients, and extensions.  
\end{proof}

\begin{subremark}
We note that \cite[Lemma 5.6.1]{Brion:StructureTheorems} shows that there exists a largest \textit{semi-abelian} $k$-subvariety of a $k$-group $G$. 
\end{subremark}

\begin{subprop}\label{prop:abelianKawamataproper}
Let $G$ be a commutative algebraic $k$-group with unipotent rank 1, and let $X$ be a closed subvariety of $G$. 
If $X$ is not fibered by subgroups, then $Z_{\A}(X) = \Exc'_{\A}(X)$ is a proper closed subset of $X$. 
\end{subprop}

\begin{proof}
By \autoref{prop:KawamataequalLanglike}, we have that $Z_{\A}(X) = \Exc'_{\A}(X)$, and hence it suffices to prove that $Z_{\A}(X)$ is a proper closed subset of $X$. 
Using \autoref{prop:largestabelian}, let $B$ be the largest abelian subvariety of $G$. 
By definition, all of the algebraic subgroups comprising $Z_{\A}(X)$ are contained in $B$. 
Consider the quotient map $\pi \colon G \to G/B$, and note that by \autoref{thm:quotientsrepn}, $G/B$ exists as a commutative algebraic $k$-group and $\pi$ is a group homomorphism.

Using Subsection \ref{subsec:equivariantcompletion}, we may form equivariant completions of $\overline{G}$  and $\overline{G/B}$, and the same argument as in \cite[Lemma 5.7]{Vojta:Integralpoints2} allow us to extend $\pi$ as a fiber bundle to a $k$-morphism
\[
\overline{\pi}\colon \overline{G} \to \overline{G/B}. 
\]
Let $\overline{X}$ be the closure of $X$ in $\overline{G}$, and let $V$ be the closed subset $\overline{\pi}(\overline{X}) \subset \overline{G/B}$. 
In this setting, one can define the Kawamata locus of $\overline{X}$ by taking the union of the Kawamata locus of each fiber over $V$,  and moreover, the intersection of this set with $X$ is precisely $Z_{\A}(X)$.

To conclude, we note that Bogomolov's proof of the Kawamata structure theorem \cite[Section 1.3 and proof of Theorem 1]{Bogomolov:PointsFiniteOrder} readily holds in this context, and this proof shows that since $X$ is not fibered by subgroups (\autoref{defn:notfibered}), $Z_{\A}(X)$ is a proper closed subset of $X$. 
\end{proof}

We are now in a position to prove \autoref{thmx:main1} and \autoref{corox:main4}. We recall the statement of these results. 

\begin{subtheorem}[= \autoref{thmx:main1}]
Let $k$ be an algebraically closed field of characteristic zero, let $G$ be a commutative algebraic group over $k$ with unipotent rank 1 and toric rank 0, and let $X$ be a closed subvariety of $G$. 
If $X$ is not fibered by subgroups, then the Kawamata locus of $X$ is a proper closed subset of $X$. 
\end{subtheorem}

\begin{proof}[Proof of \autoref{thmx:main1}]
The result follows from \autoref{lemma:unipotentKLclosed},  \autoref{lemma:toricKawamata},  and \autoref{prop:abelianKawamataproper}. 
\end{proof}

\begin{subcorollary}[= \autoref{corox:main4}]
Let $k$ be an algebraically closed field of characteristic zero, let $G$ be a commutative algebraic $k$-group with unipotent rank 1 and toric rank 0, and let $X$ be a closed subvariety of $G$. 
If $X$ is not fibered by subgroups, then $X$ is pseudo-groupless. 
\end{subcorollary}

\begin{proof}
This follows from \autoref{thmx:main0}, \autoref{thmx:main1}, and the equivalent characterization of groupless modulo $\Delta$ from \autoref{lemma:testgroupless}.(2). 
\end{proof}

\section{\bf The strong form of the Green--Griffiths--Lang--Vojta conjecture for $\mG_a \times A$}
\label{sec:NAGGLV}
In this section,  we will prove \autoref{thmx:main2} and its corollary. 
We recall the statement of this theorem for the reader's convenience. 

\begin{theorem}[= \autoref{thmx:main2}]
Let $A$ be an abelian $\mC$-variety, and let $X$ be a closed subvariety of $\mG_{a,\mC}\times A$. 
Then, $X^{\an}$ is Brody hyperbolic modulo the analytification of the Lang-like exceptional locus of $X$ (\autoref{defn:Langlike}). 
\end{theorem}

Our proof follows the same format as \cite[Theorem 1.4]{Ghiocaetal:VariantML}. 

\begin{proof}
By \autoref{thmx:main0}, it suffices to prove the following:~if there exists a Zariski-dense entire curve $\varphi\colon \mC \to X^{\an}$, then $X$ is the translate of an algebraic subgroup of $G = \mG_{a,\mC}\times A$. 
For the proof, we will drop the $\mC$ subscript from $\mG_a$ to simplify notation.

Let $\pi_2\colon \mG_a \times A \to A$ denote the second projection. 
By the Bloch--Ochiai--Kawamata theorem \cite{Ochiai77} and \cite[Theorem 3]{Kawamata}, we can replace $A$ by the Zariski closure of $\pi_2 \circ \varphi$ which will be the translate of an abelian subvariety, and hence we may and do assume that $\pi_{2|X}\colon X \to A$ is dominant. 
Let $\pi_1\colon \mG_a \times A \to \mG_a$ denote the first projection. 
If $\pi_1\circ \varphi$ is not Zariski dense, then $X$ is contained in a fiber $F\cong A$ of $\pi_1$, and via dimension reasons, we can conclude that $X$ is the translate of an abelian subvariety of $\mG_a \times A$. 
Therefore, for the remainder of the proof, we may and do assume that $\pi_{1|X}$ is dominant. 
We will prove that in this setting our assumption of a Zariski-dense entire curve will produce a contradiction. 

As in Subsection \ref{subsec:equivariantcompletion}, let $\mP^1 \times A$ be the equivariant compactification of $\mG_a \times A$, let $\overline{X}$ be the Zariski closure of $X$ in $\mP^1 \times A$, let $\overline{\pi}_1\colon \mP^1 \times A \to \mP^1$ denote the first projection, and let $\overline{\pi}_2\colon \mP^1 \times A \to A$ denote the second projection. 
Consider the normalization $\nu\colon \overline{X}'\to \overline{X}$.  
As this is a birational morphism,  \cite[Lemma 2.8.(2)]{JavaRoussea:Albanese} tells us that $\overline{X}'$ admits a Zariski-dense entire curve. 
By Stein factorization applied to the composition $\overline{\pi}_2|_{\overline{X}} \circ \nu\colon \overline{X}'\to A$, there exists a normal projective variety $Y$ with a birational morphism $\iota\colon \overline{X}' \to Y$ and a finite morphism $f\colon Y \to A$ such that $\overline{\pi}_{2}|_{\overline{X}}\circ \nu = f\circ \iota$. 
By \cite[Lemma 2.8.(2)]{JavaRoussea:Albanese}, we have that $Y$ also contains a Zariski-dense entire curve.

By the Kawamata structure theorem \cite[Theorem 13]{Kawamata:CharacterizationAV} applied to $f\colon Y \to A$, there exists a finite \'etale cover $\xi\colon \wt{Y} \to Y$ such that $\wt{Y} \cong \wt{B} \times W$ where $\wt{B}$ is a finite \'etale cover of an abelian subvariety $B \subseteq A$ and $W$ is a normal projective variety of general type i.e., $\kappa(W) = \dim(W)$. 
Since $\xi$ is finite \'etale, \cite[Proposition 2.12.(2)]{JavaRoussea:Albanese} implies that $\wt{Y}$ admits a Zariski-dense entire curve, whence $W$ also admits a Zariski-dense entire curve. 
As $Y\to A$ is finite, we have that $W\subset \wt{Y} \to Y \to A$ is a finite morphism. 
By Hironaka, we may take a resolution of singularities of $W$ to get a smooth projective variety $\wt{W}$ that is birational to $W$ and hence is general type and by \cite[Lemma 2.8.(2)]{JavaRoussea:Albanese} has a Zariski-dense entire curve. 
Now, note that $\wt{W}$ is a smooth projective variety which admits a generically finite morphism to an abelian variety $A$, and a simple dimension argument plus the universal property of Albaneses tells us that $\wt{W}$ has maximal Albanese dimension. 
However, Yamanoi's theorem (\autoref{thm:semiabsettingNA}) tells us that this can only happen if $\dim(W) = 0$. 
This forces $\wt{Y}$ to be an abelian variety, and hence $Y$ must also be an abelian variety. 

As $Y$ is birational to $X$, let $\wt{\iota}\colon X \dashrightarrow Y$ denote the corresponding dominant, birational map. 
Since $X$ admits a Zariski-dense entire curve, we have that $Y$ also  admits a Zariski-dense entire curve, namely $\wt{\iota}\circ \varphi\colon \mC \to Y^{\an}$. 
Composing the rational inverse map $\wt{\iota}^{-1}$ with $\overline{\pi}_1|_{X}$ produces a non-constant rational function $g\colon Y \dashrightarrow \mP^1$ because we work under the assumption that $\pi_{1|X}\colon X \to \mG_a$ is dominant. 
Since $g$ is non-constant,  we can consider the pole divisor\footnote{In \cite[Proof of Theorem 1.4]{Ghiocaetal:VariantML}, the authors simply take $D = g^{*}(\infty)$, but this is not the correct object to consider as $g$ is simply a rational map.  Moreover, the authors claim that since $g(P)$ is $S$-integral, $P$ is $S$-integral relative to $D$, but this is not generally true for rational maps $g$.  E.g., one may take $g\colon \mP^2 \dashrightarrow \mP^1$ given by $[x:y:z] \to [x:y]$ with $D = \{Y = 0\}$ and the points $P_{\ell} = [\ell:\ell:1]$ for any prime number $\ell$. This collection of points satisfies $g(P_{\ell}) = 1$ but there is no fixed finite set $S$ making the collection of points $\{P_{\ell}\}$ integral relative to $D$. To summarize, there appears to be a gap in their proof, which we do not know how to mend in the arithmetic setting. \label{foot}}
\[
D = \text{div}_{\infty}(g) = \sum_{E \subset Y \text{ prime divisor}}\max\{0,-\ord_E(g)\}E,
\]
which is naturally a Weil divisor on $Y$, and since $Y$ is smooth, $D$ is also a Cartier divisor.  
Note that $D \neq 0$: otherwise $g$ would have no poles and be a regular function on $Y$, and hence a constant function on $Y$ by projectivity.  
Consider the connected component of the identity of the stabilizer of $D$ in $Y$, which we denote by $\text{Stab}(D)^{\circ}$. 
Let $B$ denote the quotient abelian variety $Y / \text{Stab}(D)^{\circ}$, and let $q\colon Y \to B$ denote the quotient map.  
Because $D$ is invariant under $\text{Stab}(D)^{\circ}$, there exists an effective divisor $D_B$ on $B$ such that $D = q^*D_{B}$, and note that $D_B$ has finite stabilizer. 
By \cite[Proposition 9.1]{Milne:AbelianVarieties}, $D_B$ is ample. 

Consider $U = B\setminus \text{supp}(D_B)$, so that $q^{-1}(U) = Y\setminus \text{supp}(D)$. 
Since $Y$ is smooth, we have that $g_{|q^{-1}(U)}$ is regular and hence defines an element of $H^0(q^{-1}(U),\sO_Y)$. 
Since $q_*\sO_Y \cong \sO_B$,  we have that there is a unique $h_U \in H^0(U,\sO_B)$ such that $g_{|q^{-1}(U)} = h_U \circ q$, and by Zariski density, $h_U$ defines a non-constant rational function $h \in \mC(B)$ such that $g = h \circ q$ and $h$ is regular on $U$.  
Since $q$ is smooth and surjective, we have that $\text{div}_{\infty}(g) = q^*\text{div}_{\infty}(h)$, and hence $D_B = \text{div}_{\infty}(h)$. 

We have that $(q\circ \wt{\iota}\circ \varphi)\colon \mC \to B^{\an}$ is Zariski dense. 
To ease notation, we let $\psi \coloneqq (q\circ \wt{\iota}\circ \varphi)$. 

We claim that $\psi(\mC)\cap\, \text{supp}(D_B) \subseteq I(h)$ where $I(h)$ is the indeterminacy locus of the non-constant rational map $h$. 
Indeed, at every point of $\text{supp}(D_B)\setminus I(h)$, the map $h$ is defined and takes the value $\infty$, but the entire curve $\psi$ cannot pass through such a point by construction. 
After taking the reduced induced subscheme, we may and do assume that $D_B$ is reduced. 
Since $\codim_{B}I(h)\geq 2$,  we can use \cite[Theorem 2.5.1 \& Theorem 3.1.1]{Yamanoi:HolomorphicAV} to produce a contradiction.  
These results say that for every $\varepsilon>0$, the following inequalities of characteristic and truncated counting functions 
\[
N_{\psi}^{[1]}(r,I(h)) \leq \varepsilon T_{\psi,D_B}(r) + O(1) \quad \text{ and }\quad T_{\psi,D_B}(r) \leq N_{\psi}^{[1]}(r,D_B) + \varepsilon T_{\psi,D_B}(r) + O(1)
\]
hold for all $r>0$ outside an exceptional set of finite Lebesgue measure.  

The above claim implies that $\psi^{-1}(D_B) \subseteq \psi^{-1}(I(h))$, and therefore the above inequalities imply that for every $\varepsilon>0$,
\[
T_{\psi,D_B}(r)\leq N_{\psi}^{[1]}(r,D_B) + \varepsilon T_{\psi,D_B}(r) + O(1) \leq 2\varepsilon T_{\psi,D_B}(r) + O(1)
\]
holds outside of the same exceptional set.
Choosing $\varepsilon<1/2$ implies boundedness of $T_{\psi,D_B}(r)$ as $r\to \infty$ (again, outside of the exceptional set). 
However, this cannot happen since $D_B$ is ample and $\psi$ is nonconstant.
Therefore, we have reached a contradiction, which concludes our proof. 
\end{proof}

  \bibliography{refs.bib}{}
\bibliographystyle{amsalpha}

 \end{document}